\documentclass[11pt,a4paper,reqno]{amsart}

\usepackage{amsmath, amsthm, amsopn, amssymb}
\usepackage{bbm,cancel}
\usepackage{enumerate}
\usepackage{enumitem}
\usepackage{bbm}

\usepackage{geometry, xcolor}
\newcommand{\cB}{\mathcal{B}}
\newcommand{\cC}{\mathcal{C}}
\newcommand{\cH}{\mathcal{H}}
\newcommand{\cF}{\mathcal{F}}

\newcommand{\NN}{\mathbb{N}}

\newcommand{\TT}{\mathcal{T}}

\newcommand{\Ex}{\mathbb{E}}
\newcommand{\eps}{\varepsilon}
\newcommand{\ex}{\mathrm{ex}}
\newcommand{\exnH}{\ex(n,H)}

\newcommand{\FnH}{\mathcal{F}_n(H)}

\newcommand{\indicator}{\mathbbm{1}}

\newtheorem{thm}{Theorem}
\newtheorem{conj}[thm]{Conjecture}
\newtheorem{lemma}[thm]{Lemma}
\newtheorem{proposition}[thm]{Proposition}
\newtheorem{cor}[thm]{Corollary}

\newtheorem{question}[thm]{Question}

\theoremstyle{definition}

\renewcommand{\le}{\leqslant}
\renewcommand{\ge}{\geqslant}

\usepackage{color}

\title{Supersaturated sparse graphs and hypergraphs}

\author{Asaf Ferber}
\address{Department of Mathematics, Massachusetts Institute of Technology, 77 Massachusetts Avenue, Cambridge, MA 02139.}
\email{ferbera@mit.edu}
\thanks{Research supported in part by an NSF grant 6935855. (AF)}

\author{Gweneth Anne McKinley}
\address{Department of Mathematics, Massachusetts Institute of Technology, 77 Massachusetts Avenue, Cambridge, MA 02139}
\email{gweneth@mit.edu}

\author{Wojciech Samotij}
\address{School of Mathematical Sciences, Tel Aviv University, Tel Aviv 6997801, Israel}
\email{samotij@post.tau.ac.il}

\thanks{Research supported in part by the Israel Science Foundation grant 1147/14 (WS)}

\begin{document}

\begin{abstract}
  A central problem in extremal graph theory is to estimate, for a given graph $H$, the number of $H$-free graphs on a given set of $n$ vertices. In the case when $H$ is not bipartite, fairly precise estimates on this number are known. In particular, thirty years ago, Erd\H{o}s, Frankl, and R\"odl proved that there are $2^{(1+o(1))\exnH}$ such graphs. In the bipartite case, however, nontrivial bounds have been proven only for relatively few special graphs $H$.

  We make a first attempt at addressing this enumeration problem for a general bipartite graph $H$. We show that an upper bound of $2^{O(\exnH)}$ on the number of $H$-free graphs with $n$ vertices follows merely from a rather natural assumption on the growth rate of $n \mapsto \exnH$; an analogous statement remains true when $H$ is a uniform hypergraph. Subsequently, we derive several new results, along with most previously known estimates, as simple corollaries of our theorem. At the heart of our proof lies a general supersaturation statement that extends the seminal work of Erd\H{o}s and Simonovits. The bounds on the number of $H$-free hypergraphs are derived from it using the method of hypergraph containers.
\end{abstract}
	
\maketitle
	
\section{Introduction}

The \emph{extremal number} of a graph $H$, denoted by $\exnH$, is the maximum possible number of edges in a graph $G$ on $n$ vertices which does not contain $H$ as a (not necessarily induced) subgraph. Such a graph $G$ is referred to as $H$-\emph{free}. The study of the asymptotic behavior of $\exnH$ for various $H$ is a central theme in extremal graph theory and goes back to the pioneering work of Tur\'an~\cite{turan}, who determined $\exnH$ exactly in the case when $H$ is a complete graph. In fact, Tur\'an's construction provides a lower bound on $\exnH$ that depends on the \emph{chromatic number} of $H$, denoted by $\chi(H)$, which is the least integer $k$ for which one can partition $V(H)$ into $k$ \emph{independent sets} (that is, sets which induce no edges). More precisely, Tur\'an's construction gives
\[
  \exnH \ge \left(1-\frac{1}{\chi(H)-1}\right)\binom{n}{2}
\]
for every nonempty graph $H$. A matching upper bound was proved several years later by Erd\H{o}s and Stone \cite{EStone}, giving
\begin{equation}
  \label{eq:Erdos-Stone}
  \exnH = \left(1-\frac{1}{\chi(H)-1}\right)\binom{n}{2}+o(n^2).
\end{equation}
Note that~\eqref{eq:Erdos-Stone} determines the asymptotics of $\exnH$ whenever $\chi(H) \ge 3$, but when $\chi(H) = 2$, that is, when $H$ is \emph{bipartite}, it only implies that $\exnH = o(n^2)$, whereas Tur\'an's construction gives the trivial bound $\exnH \ge 0$.

Perhaps unsurprisingly, the bipartite case of Tur\'an's problem is much more challenging and there are only a few bipartite graphs $H$ for which even the order of magnitude of $\exnH$ has been determined. Among the known examples one can find trees, cycles of lengths four, six, and ten, and the complete bipartite graphs $K_{s,t}$ when $s\in\{2,3\}$ or $t<(s-1)!$. For a generic bipartite $H$, there does not even seem to be a good guess for what $\exnH$ might be. The lower bounds in all the above examples are established by rather involved algebraic or geometric constructions. The strongest general upper bound on $\exnH$ is due to F\"uredi~\cite{Fu91} who proved that $\exnH = O(n^{2-1/D})$ if all but one of the vertices in one of the color classes of some proper two-coloring of $H$ have degree at most $D$. This generalizes the classical result of K\H{o}v\'ari, S\'os, and Tur\'an~\cite{KoSoTu54}, who showed that $\ex(n,K_{s,t})=O(n^{2-1/s})$ for all $s$ and $t$. Treating a more general class of graphs than the one considered in \cite{Fu91}, Alon, Krivelevich, and Sudakov~\cite{AlKrSu03} proved that $\exnH = O(n^{2-1/4D})$ for every \emph{$D$-degenerate} bipartite graph $H$ (a graph is $D$-degenerate if every subgraph of it has minimum degree at most $D$). For a more detailed discussion and further references we refer the reader to the excellent survey of F\"uredi and Simonovits~\cite{FuSi13}.

Here we shall be concerned with the closely related problem of enumerating $H$-free graphs. That is, we are interested in the asymptotic size of the set $\FnH$ consisting of all (labeled) $H$-free graphs with vertex set $[n] := \{1, \dotsc, n\}$.  Observing that every subgraph of an $H$-free graph is also $H$-free and that every $n$-vertex $H$-free graph has at most $\exnH$ edges, one obtains the trivial bounds
\begin{equation}
  \label{eq:FnH-trivial}
  2^{\exnH} \le |\FnH| \le \sum_{k=0}^{\exnH} \binom{\binom{n}{2}}{k} \le \left(\frac{e\binom{n}{2}}{\exnH}\right)^{\exnH}.
\end{equation}
This counting problem has been widely studied, and when $H$ is not bipartite, bounds much tighter than \eqref{eq:FnH-trivial} are known. It was proved by Erd\H{o}s, Kleitman, and Rothschild~\cite{ErKlRo76} (when $H$ is a complete graph, but implicitly also for every non-bipartite $H$) and then by Erd\H{o}s, Frankl, and R\"odl \cite{ErFrRo86} that
\begin{equation}
  \label{eq:FnH-EFR}
  |\FnH| = 2^{\exnH + o(n^2)}.
\end{equation}
In particular, if $\chi(H) \ge 3$, then~\eqref{eq:Erdos-Stone}, ~\eqref{eq:FnH-EFR}, and the lower bound in ~\eqref{eq:FnH-trivial} imply that $|\FnH| = 2^{(1+o(1)) \exnH}$. On the other hand, if $H$ is bipartite, then~\eqref{eq:FnH-EFR} is very weak and the trivial upper bound in~\eqref{eq:FnH-trivial} is still the state-of-the-art bound for a generic graph $H$ (up to a constant multiplicative factor in the exponent), giving
\begin{equation}
  \label{eq:FnH-bip-trivial}
  2^{\exnH} \le |\FnH| \le 2^{C \exnH \log n}
\end{equation}
for some positive constant $C$ that depends only on $H$. It is natural to ask whether the $\log n$ factor in the above upper bound can be removed. Indeed, this question was posed by Erd\H{o}s some thirty five years ago (see~\cite{KlWi82}) for all bipartite $H$ that contain a cycle.\footnote{The case when $H$ has no cycles is very different as then $\exnH = O(n)$ while there could be $n$-vertex $H$-free graphs with  as many as $n!$ different labelings. In particular, since there are $2^{\Omega(n \log n)}$ different labeled $n$-vertex graphs with maximum degree one, then $|\FnH| \ge 2^{\Omega(\exnH \cdot \log n)}$ for every acyclic $H$ with maximum degree at least two. Worse still, there are $2^{\Omega(n \log n)}$ non-isomorphic $n$-vertex graphs with maximum degree three.} Until very recently, it was even believed that the stronger bound $|\FnH| \le 2^{(1+o(1)) \exnH}$ holds, as it does for non-bipartite $H$, but this was disproved by Morris and Saxton~\cite{MS} in the case when $H$ is the cycle of length six. In view of this, the following seems to be the right question to ask.

\begin{question}
  \label{question:main}
  Suppose that $H$ is a bipartite graph which contains a cycle. Is there a constant $C$ such that
  \[
    |\FnH| \le 2^{C \exnH}
  \]
  for all $n$?
\end{question}

Despite renewed interest in Question~\ref{question:main} in recent years, very little is known. To the best of our knowledge, it has been answered positively only in the cases when $H$ is the cycle of length four~\cite{KlWi82}, six~\cite{KlWi96}, or ten~\cite{MS}, the complete bipartite graph $K_{s,t}$ with $s \in \{2, 3\}$ or $t > (s-1)!$ (see~\cite{BaSaKmm, BaSa11}), or so-called theta-graphs \cite{Tran}. In this paper we make a first attempt at addressing Question~\ref{question:main} for a generic bipartite graph $H$. Our methods also extend to the setting of uniform hypergraphs, which we shall discuss at the end of this section. The following is our first main result:

\begin{thm}
  \label{thm:main-graphs}
  Let $H$ be an arbitrary graph containing a cycle. Suppose that there are positive constants $\alpha$ and $A$ such that $\exnH \le A n^{\alpha}$ for all $n$. Then there exists a constant $C$ depending only on $\alpha$, $A$, and $H$ such that for all $n$,
  \[
    |\FnH| \le 2^{C n^\alpha}.
  \]
\end{thm}

Note that Theorem~\ref{thm:main-graphs} answers Question~\ref{question:main} in the affirmative for every bipartite $H$ such that $\exnH = \Theta(n^{\alpha})$ for some $\alpha$. This is the case for each $H$ for which Question~\ref{question:main} has been answered so far and therefore Theorem \ref{thm:main-graphs} reproves all the previously known results listed above. In fact, it is commonly believed that $\exnH=\Theta(n^\alpha)$ for all bipartite $H$, as conjectured by Erd\H{o}s and Simonovits (see for example~\cite{Er81}):

\begin{conj}
  \label{conj:rational}
  For every nonempty bipartite graph $H$, there exist a rational number $\alpha \in [1,2)$ and $c>0$ such that
  \[
    \frac{\exnH}{n^{\alpha}}\rightarrow c.
  \]
\end{conj}

Observe that if Conjecture \ref{conj:rational} is true, then Theorem~\ref{thm:main-graphs} resolves Question~\ref{question:main} for all $H$. Actually, the following weaker version of Conjecture \ref{conj:rational} is sufficient. However, a solution to either of these conjectures is most likely unattainable in the near future.

\begin{conj}
  \label{conj:weak}
  For every nonempty bipartite graph $H$, there exist $\alpha \in [1,2]$ and $c_2>c_1>0$ such that
  \[
    c_1\le \frac{\exnH}{n^{\alpha}} \le c_2
  \]
\end{conj}

On a related note, we would like to mention a recent breakthrough of Bukh and Conlon \cite{BuCo}, who used a random algebraic method, pioneered by Bukh \cite{bukh2015random}, to prove  the following `inverse' version of Conjecture~\ref{conj:weak}: for every rational $\alpha\in [1,2)$, there exists a finite family of graphs $\mathcal L$ for which $\ex(n,\mathcal L)=\Theta(n^{\alpha})$ (where $\ex(n,\mathcal L)$ is the maximum possible number of edges in an $n$-vertex graph that does not contain any member of the family $\mathcal L$).

There are bipartite graphs $H$ for which the best known upper bound on $\exnH$ is of the form $O(n^{\alpha})$, for some explicit $\alpha$, and is conjectured to be tight. For such graphs, it makes sense to establish the bound $|\FnH|\le 2^{O(n^{\alpha})}.$ Indeed, such results have been proved for even cycles~\cite{MS}, complete bipartite graphs~\cite{BaSaKmm, BaSa11}, and theta graphs~\cite{Tran}. All these estimates follow as simple corollaries of Theorem~\ref{thm:main-graphs} and the corresponding upper bounds on the extremal numbers \cite{BoSi74, FaSi83, KoSoTu54}.

Even if the asymptotic behavior of $\exnH$ is unknown, assuming a sufficiently strong lower bound on it, in Theorem~\ref{thm:infiniteSequence}, we are able to prove strong estimates for $|\FnH|$ for an infinite sequence of $n$. A similar result for the number of $k$-arithmetic-progression-free subsets of $[n]$ was obtained by Balogh, Liu, and Sharifzadeh~\cite{BLS}. This result served as an inspiration for our work. Before formally stating the theorem, we recall the notion of \emph{$2$-density} of a graph $H$:
\[
  m_2(H):=\max\left\{\frac{e_F-1}{v_F-1} \colon F\subseteq H, v_F>2\right\}.
\]

\begin{thm}\label{thm:infiniteSequence}
  Let $H$ be a graph and assume that $\exnH \ge \eps n^{2-1/m_2(H)+\eps}$ for some $\eps > 0$ and all $n$. Then there exist a constant $C$ depending only on $\eps$ and $H$ and an infinite sequence of $n$ for which
  \[
    |\FnH|\le 2^{C\cdot \exnH}.
  \]
\end{thm}

The assumption on $H$ stated in Theorem~\ref{thm:infiniteSequence} is widely believed to hold for every $H$ containing a cycle. In fact, it is known to hold for quite a few bipartite graphs. For example, it is known that for every $\ell$,
\[
  \ex(n,C_{2\ell}) \ge \Omega\left(n^{1 + \frac{2}{3\ell+3}}\right) = \Omega\left(n^{2 - 1/m_2(C_{2\ell}) + \eps_\ell}\right),
\]
where $\eps_\ell > 0$; see, for example, Terlep and Williford~\cite{TeWi12} and the references therein (in particular, the famous papers of Margulis~\cite{Ma88} and Lubotzky, Phillips, and Sarnak~\cite{LuPhSa88}). To give another example, consider the case when $H$ is the $3$-dimensional hypercube graph $Q_3$. Theorem~\ref{thm:infiniteSequence} applies to $H$ because $2 - 1/m_2(Q_3) = 2 - 6/11 < 3/2$ and $\ex(n,Q_3) \ge \ex(n,C_4) = \Omega(n^{3/2})$. As a third example, note that $\ex(n,K_{4,4}) \ge \ex(n,K_{3,3}) = \Omega(n^{5/3})$ and $5/3 > 2 - 7/15 = 2 - 1/m_2(K_{4,4})$ and thus Theorem~\ref{thm:infiniteSequence} also applies with $H = K_{4,4}$. Finally, it follows from the work of Ball and Pepe~\cite{BaPe12} that $K_{5,5}$ and $K_{6,6}$ also satisfy the assumptions of Theorem~\ref{thm:infiniteSequence}.


One may consider a natural extension of Question~\ref{question:main} to the setting of uniform hypergraphs, where $\exnH$ and $\FnH$ are defined in the obvious way. However, the problem of enumerating hypergraphs without a forbidden subhypergraph has only been addressed fairly recently. Generalizing~\eqref{eq:FnH-EFR}, Nagle, R\"odl, and Schacht~\cite{NaRoSc06} proved that for each $r$-uniform hypergraph $H$,
\begin{equation}
  \label{eq:FnH-NRS}
  |\FnH| = 2^{\exnH+o(n^r)}.
\end{equation}
Analogously to the graph case, it is easy to see that an $r$-uniform hypergraph $H$ that is not $r$-partite\footnote{An $r$-uniform hypergraph $H$ is $r$-partite if its vertex set admits a partition into $r$ parts such that every edge of $H$ contains one vertex from each of the parts.} satisfies $\exnH = \Omega(n^r)$. On the other hand, extending the result of K\H{o}v\'ari, S\'os, and Tur\'an~\cite{KoSoTu54} to hypergraphs, Erd\H{o}s~\cite{Er64} proved that for every $r$-partite $r$-uniform $H$, there is an $\eps > 0$ such that $\exnH = O(n^{r-\eps})$. In particular, \eqref{eq:FnH-NRS} implies that $|\FnH| = 2^{(1+o(1)) \exnH}$ for all non-$r$-partite $r$-uniform $H$, but it gives a very weak bound when $H$ is $r$-partite. Therefore, the right generalization of Question~\ref{question:main} to the setting of hypergraphs with uniformity larger than two seems to be the following:

\begin{question}
  \label{question:main-hypergraphs}
  Suppose that $r \ge 3$ and suppose that $H$ is an $r$-partite $r$-uniform hypergraph. Under what conditions can one expect the existence of a constant $C$ such that
  \[
    |\FnH| \le 2^{C \exnH}
  \]
  for all $n$?
\end{question}
As mentioned above, our proof method applies to hypergraphs and both Theorems~\ref{thm:main-graphs} and~\ref{thm:infiniteSequence} extend to this setting. Before stating them formally, we need the following definition, which generalizes the notion of $2$-density to hypergraphs. The \emph{$r$-density} of an $r$-uniform hypergraph $H$, denoted by $m_r(H)$, is defined by
\[
  m_r(H)=\max\left\{\frac{e_F-1}{v_F-r} \colon F\subseteq H,\text{ }v_F>r \right\}.
\]
The hypergraph analog to Theorem \ref{thm:main-graphs} is the following:
\begin{thm}
  \label{thm:main}
  Let $H$ be an $r$-uniform hypergraph and let $\alpha$ and $A$ be positive constants. Suppose that $\alpha > r-1/m_r(H)$ and that $\exnH \le A n^{\alpha}$ for all $n$. Then there exists a constant $C$ depending only on $\alpha$, $A$, and $H$ such that for all $n$,
  \[
    |\FnH| \le 2^{C n^\alpha}.
  \]
\end{thm}

The idea of investigating Question~\ref{question:main-hypergraphs} was suggested in a recent work of Mubyai and Wang~\cite{MuWa}. They conjectured that Question~\ref{question:main-hypergraphs} has an affirmative answer in the case when $H$ is $C_k^{(r)}$, the $r$-uniform expansion\footnote{ Given a graph $G$ and an integer $r \ge 3$, we define the $r$-uniform \emph{expansion} of $G$ to be the hypergraph $G^{(r)}$ with edge set $\{e \cup S_e \colon e \in E(G)\}$, where $\{S_e\}_{e \in E(G)}$ are pairwise disjoint $(r-2)$-element sets disjoint from $V(G)$.} of $C_k$, the ($2$-uniform) cycle of length $k$. Improving upon the result from~\cite{han2017hypergraphs,MuWa}, Balogh, Narayanan, and Skokan~\cite{BaNaSk} have recently solved the conjecture of Wang and Mubayi. As immediate corollaries from Theorem~\ref{thm:main} we reprove this result along with two related estimates for expansions of paths and complete bipartite graphs. For further reading about Tur\'an problems for graph expansions, we refer the reader to a recent survey of Mubayi and Verstra\"ete~\cite{MuVe16} and the references therein. Here is a summary of our results:


%


\begin{cor}
  Suppose that $H$ is any one of the following:
  \begin{enumerate}
  \item
    $P_k^{(r)}$ for some $k, r \ge 3$, or
  \item
    $C_k^{(r)}$ for some $k, r \ge 3$, or
  \item
    $K_{s,t}^{(3)}$ for some $s, t \ge 3$ with $t > (s-1)!$.
  \end{enumerate}
  Then, there exists a constant $C$ depending only on $H$ such that for all $n$,
  \[
    |\FnH| \le 2^{C \cdot \exnH}.
  \]
\end{cor}

We conclude with the following analog of Theorem~\ref{thm:infiniteSequence} in the hypergraph setting.

\begin{thm}
  \label{thm:infiniteSequenceHyper}
  Let $H$ be an $r$-uniform hypergraph and assume that $\exnH \ge \eps n^{r-1/m_r(H)+\eps}$ for some $\eps > 0$ and all $n$. Then there exist a constant $C$ depending only on $\eps$ and $H$ and an infinite sequence of $n$ for which
  \[
    |\FnH|\le 2^{C\cdot \exnH}.
  \]
\end{thm}

All of our theorems are obtained as (more or less) simple corollaries of the more general but somewhat technical Theorem~\ref{thm:mainNew}, which is stated (and proved) in Section~\ref{sec:supersaturation}.

The rest of the paper is organized as follows: First, in Section~\ref{sec:discussion} we give a short discussion of our proof method, including some comments about previous work. Then, in Section~\ref{sec:containers}, we present the main tool to be used in our proofs, Lemma~\ref{lem:containers}, which is a version of a similar lemma from~\cite{MS} and is based on the method of hypergraph containers developed in~\cite{BMS,ST}. Next, in Section~\ref{sec:supersaturation}, we introduce our main technical theorem, Theorem~\ref{thm:mainNew}, a `balanced supersaturation' result that complements Lemma~\ref{lem:containers}. Next, in Sections~\ref{sec: main theorems} and~\ref{thm:infiniteSequenceHyper}, we prove Theorems~\ref{thm:main-graphs} and~\ref{thm:main} and Theorems~\ref{thm:infiniteSequence} and~\ref{thm:infiniteSequenceHyper}, respectively. Finally, in Section~\ref{sec:remarks}, we close this paper with some concluding remarks and a discussion of future research directions.

\section{Discussion}
\label{sec:discussion}

As we have mentioned in the introduction, enumeration problems in the context of forbidden (hyper)graphs have been successfully addressed for non-bipartite graphs~\cite{BaBoSi04, ErFrRo86, ErKlRo76} and non-$r$-partite $r$-uniform hypergraphs~\cite{NaRoSc06}. A main difficulty in extending the results of~\cite{BaBoSi04, ErFrRo86, ErKlRo76} to the bipartite case is that the proofs in~\cite{BaBoSi04, ErFrRo86} are based on Szemer\'edi's regularity lemma. Even though there are now sparse versions of the regularity lemma, it is unlikely that the regularity approach could be used for counting graphs without a bipartite subgraph. The proof method of~\cite{ErKlRo76} is different, but it hinges on the fact that for non-bipartite $H$, the number of edges in most graphs in $\FnH$ is $n^{2-o(1)}$; this is no longer true when $H$ is bipartite. In the case of $r$-uniform hypergraphs ($r\ge 3$), the situation is even more complicated, as a hypergraph regularity lemma which is sufficiently strong to address the enumeration problem was proved only relatively recently and is quite involved.
	
A nowadays standard way of tackling enumeration problems of this type is by using the method of \emph{hypergraph containers}. This method was introduced by Balogh, Morris, and Samotij~\cite{BMS} and, independently, by Saxton and Thomason~\cite{ST}. In particular, it can be used to reprove~\eqref{eq:FnH-NRS} for all $r$-uniform $H$ in a simple way. The container method essentially reduces the problem of establishing upper bounds on $|\FnH|$ to proving the following statement: If an $n$-vertex graph contains `slightly more' than $\exnH$ edges, then it has `many' copies of $H$ (such property is known as \emph{supersaturation}) that are moreover `well-distributed'.

Keeping this in mind, it seems hopeless to provide a general solution to the counting problem, as it seems crucial to know the order of magnitude of $\exnH$ in order to establish a sufficiently strong supersaturation result. However, Balogh, Liu, and Sharifzadeh~\cite{BLS} have recently managed to settle a question that has a similar flavor without knowing the corresponding extremal function. Specifically, they showed that for infinitely many $n$, there are $2^{\Theta(\Gamma_k(n))}$ many subsets of $[n]$ that do not contain an arithmetic progression of length $k$; here $\Gamma_k(n)$ is the largest cardinality of a subset of $[n]$ without a $k$-term arithmetic progression. We have found this result very surprising, as the asymptotic behavior of $\Gamma_k(n)$ is unknown. It motivated us to investigate whether similar estimates can be obtained for the problem of counting $H$-free graphs. A fact that was crucially used in~\cite{BLS} is that every pair of integers is contained in a constant number of $k$-term arithmetic progressions. This is not the case with copies of a fixed graph $H$ in a large complete graph (and pairs of edges of this complete graph) and this was one of the main challenges that we had to overcome.
	
The main contribution of this work is a general supersaturation theorem for $r$-uniform $r$-partite hypergraphs, Theorem~\ref{thm:mainNew} below. Roughly speaking, it states the following. Suppose that $\exnH = O(n^\alpha)$ for some $\alpha$ such that the expected number of copies of (the densest subgraph of) the forbidden hypergraph $H$ in the random hypergraph with $n$ vertices and $n^\alpha$ edges is of larger order of magnitude than $n^\alpha$. Then every $n$-vertex hypergraph with at least $n^\alpha$ edges contains `many' copies of $H$ which are `well-distributed'. Although the number of copies of $H$ that we can guarantee is still very far from the value conjectured by Erd\H{o}s and Simonovits~\cite{ErSi84}, the lower bound we prove for this quantity is sufficiently strong to allow us to derive a strong upper bound on $\FnH$ using the container method. This was in fact already observed by Morris and Saxton~\cite{MS}, who formulated the following conjecture and showed that it implies a positive answer to Question~\ref{question:main}. For an $r$-uniform hypergraph $\cH$ and $1\le \ell \le r$, let $\Delta_{\ell}(\cH)$ be the maximum number of hyperedges of $\cH$ that contain a given set of $\ell$ vertices.

\begin{conj}
  [{\cite[Conjecture 1.6]{MS}}]
  \label{conj:MorrisSaxton}
  Given a bipartite graph $H$, there exist constants $C > 0$, $\eps > 0$, and $k_0\in \NN$ such that the following holds. Let $k\ge k_0$ and suppose that $G$ is a graph on $n$ vertices with $k\cdot \exnH$ edges. Then there exists a (non-empty) collection $\cH$ of copies of $H$ in $G$, satisfying
  \[
    \Delta_{\ell}(\cH)\le \frac{C\cdot e(\cH)}{k^{(1+\varepsilon)(\ell-1)}} \text{ for all }1\le \ell\le e_H.
  \]
\end{conj}

Although we have not succeeded in resolving Conjecture~\ref{conj:MorrisSaxton}, our Theorem~\ref{thm:mainNew} shows that the `balanced supersaturation' property asserted by it holds for every graph $H$ for which Conjecture~\ref{conj:weak} (or the stronger Conjecture~\ref{conj:rational}) is true.

\section{A container lemma}
\label{sec:containers}

Let $H$ be an $r$-uniform hypergraph and let $\cH$ denote the $e_H$-uniform hypergraph whose vertex set is the edge set of the complete $r$-uniform $n$-vertex hypergraph~$K_n^{(r)}$ and whose hyperedges are (the edge sets of) all copies of $H$ in $K_n^{(r)}$. Note that the edge set of every $H$-free hypergraph on $n$ vertices corresponds to an \emph{independent set} in $\cH$ and vice versa. Therefore, any upper bound on the number of independent sets in $\cH$ yields an upper bound on the number of $H$-free hypergraphs.

In order to obtain the desired bound on the number of independent sets, we will use a version of the container lemma due to Balogh, Morris, and Samotij~\cite[Proposition~3.1]{BMS}. Roughly speaking, the lemma states that if the edges of a uniform hypergraph $\cH$ are `well-distributed', then the following holds. There is a `relatively small' collection $\cC$ of subsets of $V(\cH)$ (referred to as \emph{containers}), each of which induces `not too many' hyperedges, such that every independent set of $\cH$ is a subset of at least one container. Here is the formal statement:

\begin{proposition}[Container lemma {\cite[Proposition 3.1]{BMS}}]
  \label{prop:containers}
  Let $\cH$ be a $k$-uniform hypergraph and let $K$ be a constant. There exists a constant $\delta$ depending only on $k$ and $K$ such that the following holds. Suppose that for some $p\in(0,1)$ and all $\ell \in \{1, \dotsc, k\}$,
  \begin{equation}
    \label{eq:pseudorandom}
    \Delta_{\ell}(\cH) \le K\cdot p^{\ell-1}\cdot\frac{e(\cH)}{v(\cH)}.
  \end{equation}
  Then, there exists a family $\cC\subseteq \mathcal P(V(\cH))$ of \emph{containers} with the following properties:
  \begin{enumerate}[label={(\roman*)}]
  \item
    $|\cC| \le \binom{v(\cH)}{\le k p v(\cH)} \le \left(\frac{e}{kp}\right)^{k p v(\cH)}$,
  \item
    $|G| \le (1-\delta) \cdot v(\cH)$ for each $G \in \cC$,
  \item
    each independent set of $\cH$ is contained in some $G \in \cC$.
  \end{enumerate}
\end{proposition}

Clearly, the smaller the $p$ we choose, the stronger the upper bound on the number of containers. On the other hand, as we decrease $p$, it becomes more difficult to satisfy the `density' condition \eqref{eq:pseudorandom}.

To illustrate how the container lemma can be applied in our setting, let us assume that we have an upper bound of $O(M)$ on the largest size of a container and that~\eqref{eq:pseudorandom} is fulfilled with $p$ satisfying $p\log \frac 1p =O(M/v(\cH))$. Then, we immediately obtain
\[
  |\FnH|\le |\cC|\cdot 2^{O(M)}=2^{O(M)}.
\]
Since one does not obtain strong bounds on the largest size of a container after one application of Proposition~\ref{prop:containers}, it is natural to iterate it. Specifically, given a candidate $G$ for a final container, we can either decide to keep it (if $G$ is small enough for our purposes) or invoke Proposition~\ref{prop:containers} to the induced subhypergraph $\cH[G]$ to break $G$ down further. In order for this recursive process not to produce too many containers, we must prove that $\cH[G]$ fulfills~\eqref{eq:pseudorandom} with a `relatively small' $p$. Unfortunately, since we do not know anything about the structure of $G$, such a statement might be very hard, or even impossible to prove.

In order to overcome this difficulty, we employ the following simple, yet powerful strategy that was first used in this context by Morris and Saxton \cite{MS}. Given any subhypergraph $\cH_G \subseteq \cH[G]$, every independent set in $\cH[G]$ is also independent in $\cH_G$. Hence, any upper bound on the number of independent sets in $\cH_G$ is also an upper bound on the number of independent sets in $\cH[G]$. It thus follows that even if $\cH[G]$ does not fulfill \eqref{eq:pseudorandom}, we might hope to find a suitable subhypergraph $\cH_G\subseteq \cH[G]$ which does satisfy this condition, enabling us to continue the iteration.

With this strategy in mind, we are first going to show how the existence of such $\cH_G$ for every $G$ implies the desired upper bound on the number of independent sets. A similar statement appears in~\cite{MS}, but since we consider hypergraphs here as well (as opposed to~\cite{MS}), for the convenience of the reader and in order to keep this paper self-contained, we include a full proof.

\begin{lemma}
  \label{lem:containers}
  Let $H$ be a nonempty $r$-uniform hypergraph and let $\cH$ be the $e_H$-uniform hypergraph comprising (the edge sets of) all copies of $H$ in $K_n^{(r)}$. Let $K$ be a constant and let $\gamma=\frac{1}{1-\delta}$, where $\delta:=\delta(e_H, K)$ is defined in Proposition~\ref{prop:containers}. Suppose that for a given $n\in\NN$, there exist $M$ and $t_0$ such that the following holds: for all integers~$t \ge t_0$ and all $G \subseteq V(\cH)$ satisfying
  \[
    \gamma^t M < |G|\le \gamma^{t+1} M,
  \]
  there exists a subhypergraph $\cH_G\subseteq \cH[G]$ for which
  \begin{equation}
    \label{eq:container-hypotheses}
    \Delta_\ell(\cH_G)\le K\cdot \left(\frac{b_t}{|G|}\right)^{\ell -1}\cdot \frac{e(\cH_G)}{|G|}
  \end{equation}
   where $b_t = \frac{M}{(t+1)^3}$,
  for all  $\ell\in\{1, \dotsc, e_H\}$. Then there is a constant $C$ depending only on $K$, $t_0$, and $e_H$ such that $|\FnH| \le 2^{C\cdot M}$.
\end{lemma}

\begin{proof}
  We are going to prove the claimed upper bound on $|\FnH|$ by constructing a collection of $2^{O(M)}$ containers for independent sets in $\cH$, each of size $O(M)$. We start with the trivial container $V(\cH)$ which we break down into smaller containers by repeatedly applying Proposition~\ref{prop:containers} to the subhypergraphs $\cH_G$ from the assumption of the lemma. Formally, we shall construct a rooted tree $\TT$ whose vertices are subsets of $V(\cH)$, that is, subgraphs of $K_n^{(r)}$, with the following properties:
  \begin{enumerate}[label=(T\arabic*)]
  \item
    The root of $\TT$ is $V(\cH)$.
  \item
    \label{item:T2}
    If $G$ is a non-leaf vertex of $\TT$, then every independent set of $\cH[G]$ is an independent set of $\cH[G']$ for some child $G'$ of $G$ in $\TT$.
  \item
    Every leaf of $\TT$ is a subset of $V(\cH)$ with at most $\gamma^{t_0} M$ elements.
  \end{enumerate}
  The existence of such a tree $\TT$ clearly implies that
  \begin{equation}
    \label{eq:FnH-leaves-TT}
    |\FnH| \le \text{\#leaves of $\TT$} \cdot 2^{\gamma^{t_0} M}.
  \end{equation}
  We construct $\TT$ greedily by starting from a tree comprising just the root $V(\cH)$ and repeatedly `splitting' every leaf vertex that corresponds to a subset of $V(\cH)$ with more than $\gamma^{t_0} M$ elements. Suppose that $G$ is such a subset and let $t \ge t_0$ be the unique integer such that
  \begin{equation}
    \label{eq:t-G-def}
    \gamma^t M < |G| \le \gamma^{t+1} M.
  \end{equation}
  By our assumption, there is a subhypergraph $\cH_G\subseteq \cH[G]$ that satisfies condition~\eqref{eq:container-hypotheses}. Observe that if we let $p=\frac{b_t}{|G|}$, then we obtain precisely~\eqref{eq:pseudorandom}. Therefore, we can apply Proposition~\ref{prop:containers} to $\cH_G$ and obtain a family $\cC_G$ of subsets of $G$ such that
  \begin{enumerate}[label=(\roman*)]
  \item
    $|\cC_G| \le \left(\frac{e |G|}{e_H b_t}\right)^{e_H b_t} \le \left(\frac{e \gamma^{t+1} M}{e_H b_t}\right)^{e_H b_t}$,
  \item
    \label{item:G-split-2}
    $|G'| \le (1-\delta) \cdot |G| \le \gamma^t M$ for every $G' \in \cC_G$,
  \end{enumerate}
  and such that~\ref{item:T2} holds for $G$, as every independent set in $\cH[G]$ is still independent in $\cH_G$. Note that~\ref{item:G-split-2} implies that as $G$ ranges over the vertices of any path from the root to a leaf of $\TT$, the sequence of $t$ satisfying~\eqref{eq:t-G-def} is strictly decreasing. Moreover, $t \le T$, where $T$ is the smallest integer satisfying $\gamma^T M > v(\cH)$. It follows that
  \begin{equation}
    \label{eq:leaves-T}
    \begin{split}
      \text{\#leaves of $\TT$} & \le \prod_{t = t_0}^T \left(\frac{e \gamma^{t+1} M}{e_H b_t}\right)^{e_H b_t} \le \prod_{t=t_0}^T \left(\frac{e\gamma^{t+1}(t+1)^3}{e_H}\right)^{\frac{e_H M}{(t+1)^3}} \le \prod_{t=t_0}^T \left(A^{t+1}\right)^{\frac{e_H M}{(t+1)^3}} \\
      & \le \exp\left(e_H M \cdot \log A \cdot \sum_{t = 1}^{\infty} \frac{1}{t^2} \right) \le 2^{(C-\gamma^{t_0}) M},
    \end{split}
  \end{equation}
  where $A$ and $C$ are constants depending only on $\gamma$, $e_H$, and $t_0$. The assertion of the lemma now follows from~\eqref{eq:FnH-leaves-TT} and~\eqref{eq:leaves-T}.
\end{proof}

\section{Supersaturation}
\label{sec:supersaturation}

In this section we establish our supersaturation statement for copies of a fixed hypergraph $H$. We shall be able to prove, for every $n$-vertex hypergraph $G$, the existence of an $\cH_G$ as in the discussion before Lemma~\ref{lem:containers} using only a relatively mild and natural assumption on the growth rate of $\ex(s, H)$ for all $s$ below some given $n$. As in the argument of~\cite{MS}, we build $\cH_G$ by adding suitable copies of $H$ in $G$ one by one. The following technical statement is the main contribution of our work. The key idea in its proof, a double counting argument based on averaging over induced subhypergraphs of $G$, can be traced back to the seminal work of Erd\H{o}s and Simonovits~\cite{ErSi83}.

\begin{thm}
  \label{thm:mainNew}
  Let $H$ be an $r$-uniform hypergraph, let $\gamma > 1$, and let $\alpha > r - 1/m_r(H)$. Suppose that $M$ is such that for every $s \in \{1, \dotsc, n\}$,
  \[
    \ex(s,H) \le M \cdot \left(\frac{s}{n}\right)^\alpha.
  \]
  Then there exists a constant $t_0$ depending only on $\alpha$, $\gamma$, and $H$ such that the following holds.
  If $G$ is an $n$-vertex $r$-uniform hypergraph with
  \[
    \gamma^t M < e(G) \le \gamma^{t+1} M
  \]
  for some integer $t \ge t_0$, then there is a collection $\cH_G$ of copies of $H$ in $G$ for which, letting $b_t = \frac{M}{(t+1)^3}$,
  \begin{equation}
    \label{eq:Delta-ell-condition}
    \Delta_\ell(\cH_G) \le 2^{2e_H+3} \cdot \left(\frac{b_t}{e(G)}\right)^{\ell-1} \cdot \frac{e(\cH_G)}{e(G)}
  \end{equation}
  for every $\ell\in \{1, \dotsc, e_H\}$. In particular, Lemma~\ref{lem:containers} implies the existence of a constant $C$ depending only on $\alpha$ and $H$ such that $|\cF_n(H)| \le  2^{C\cdot M}$.
\end{thm}

\begin{proof}
  Let $\cH$ denote the hypergraph with vertex set $E(G)$ comprising all copies of $H$ in $G$. We shall construct an $\cH_G \subseteq \cH$ from an initially empty hypergraph by adding to it copies of $H$ one by one, in a sequence of $N$ steps ($N$ to be chosen shortly). We shall do it in such a way that after $N$ steps, the obtained hypergraph $\cH_G$ will have exactly $N$ edges and will satisfy~\eqref{eq:Delta-ell-condition}.

  Let $m = e(G)$. Since we will add each copy of $H$ to $\cH_G$ only once, we will have $\Delta_{e_H}(\cH_G) = 1$ and thus, isolating $e(\cH_G)$ in~\eqref{eq:Delta-ell-condition} with $\ell = e_H$, the number of edges that we have to add to $\cH_G$ satisfies
  \[
    N \ge \left(\frac{m}{b_t}\right)^{e_H-1}\cdot 2^{-2e_H-3} \cdot m.
  \]
  In particular, choosing
  \[
    N := \left(\frac{\gamma^{t+1} M}{b_t}\right)^{e_H-1} \cdot m = \left((t+1)^3 \cdot \gamma^{t+1}\right)^{e_H-1} \cdot m,
  \]
  we will guarantee that~\eqref{eq:Delta-ell-condition} holds for $\ell = e_H$.

  We now make the above discussion precise. We shall construct a sequence $(\cH_i)_{i=0}^N$ of subhypergraphs of $\cH$ such that $\cH_i \subseteq \cH_{i+1}$ and $e(\cH_i) = i$ for each $i$ and let $\cH_G = \cH_N$. We let $\cH_0$ be the empty hypergraph. Suppose that we have already defined $\cH_i$ for some $i \in \{0, \dotsc, N-1\}$. Our goal is not only to find some copy of $H$ in $\cH \setminus \cH_i$ to be added to $\cH_i$ in order to form $\cH_{i+1}$, but also to choose this copy carefully so that at the end of the process, condition~\eqref{eq:Delta-ell-condition} is satisfied for every $\ell$. To this end, for every nonempty $F \subsetneq H$, we let $\cB_F(\cH_i)$ denote the collection of `bad' copies of $F$ in $G$ in the sense that they are already `saturated' in $\cH_i$. That is, the $\cH_i$-degree of the set of $e(F)$ edges of $G$ that form this copy of $F$ is close to violating the bound~\eqref{eq:Delta-ell-condition}, with $\ell = e(F)$. More precisely, given $F' \subseteq V(\mathcal H_i)$, we define
  \[
    \deg_{\cH_i}F' = \left|\left\{E\in E(\mathcal H_i)\colon F'\subseteq E\right\}\right|,
  \]
  and let
  \[
    \cB_F(\cH_i)=\left\{ F' \subseteq G \colon F' \simeq F \text{ and } \deg_{\cH_i} F' \ge 2^{2e_H+2} \cdot \left((t+1)^3 \cdot \gamma^{t+1}\right)^{1-e_F} \cdot \frac{N}{m} \right\}.
  \]
  Observe that
  \[
    2^{e_H} \cdot N \ge \binom{e_H}{e_F} \cdot e(\cH_i) \ge |\cB_F(\cH_i)| \cdot 2^{2e_H+2} \cdot \left((t+1)^3 \cdot \gamma^{t+1}\right)^{1-e_F} \cdot \frac{N}{m}
  \]
  and therefore,
  \begin{equation}
    \label{eq:B-F-upper}
    |\cB_F(\cH_i)| \le 2^{-e_H-2} \left((t+1)^3 \cdot \gamma^{t+1}\right)^{e_F-1} \cdot m.
  \end{equation}
  Suppose that there exists an $E \in \cH$ such that $F' \notin \cB_F(\cH_i)$ for every nonempty $F \subsetneq H$ and every $F \simeq F' \subsetneq E$. Call each such $E$ \emph{good}, assuming that $i$ is fixed. If there is a good $E$ that is not already in $\cH_i$, then letting $\cH_{i+1} = \cH_i \cup \{E\}$ guarantees that for every $\ell \in [e_H-1]$,
  \[
    \begin{split}
      \Delta_\ell(\cH_{i+1}) & \le \max\left\{ \Delta_\ell(\cH_i), \; 2^{2e_H+2} \cdot \left((t+1)^3 \cdot \gamma^{t+1}\right)^{1-\ell} \cdot \frac{N}{m} + 1\right\} \\
      & \le \max\left\{ \Delta_\ell(\cH_i), \; 2^{2e_H+3} \cdot \left((t+1)^3 \cdot \gamma^{t+1}\right)^{1-\ell} \cdot \frac{N}{m} \right\} \\
      & \le \max\left\{ \Delta_\ell(\cH_i), \; 2^{2e_H+3} \cdot \left(\frac{b_t}{m}\right)^{\ell-1} \cdot \frac{N}{m} \right\},
    \end{split}
  \]
  where the second inequality holds because
  \[
    \left((t+1)^3 \cdot \gamma^{t+1}\right)^{1-\ell} \cdot \frac{N}{m} = \left((t+1)^3 \cdot \gamma^{t+1}\right)^{e_H-\ell} \ge 1
  \]
  and the last inequality uses the definition of $b_t$ and the bound $m \le \gamma^{t+1} M$. In particular, by the definition of $N$, if we succeed in finding such a good $E \in \cH \setminus \cH_i$ for every $i$, then the final hypergraph $\cH_G = \cH_N$ will satisfy~\eqref{eq:Delta-ell-condition} for every $\ell \in [e_H]$.

  Fix some $p \in (0, 1]$ such that $p n$ is an integer and let $R$ be a uniformly chosen random subset of $p n$ vertices of $G$. Denote by $G'$ the subgraph of $G$ induced by $R$. Let $G''$ be a graph obtained from $G'$ by removing one edge from each copy of $F$ in $G'$ that belongs to $\cB_F(\cH_i)$, for every nonempty $F \subsetneq H$. Note that any copy of $H$ in $G''$ is good by definition. Let $X$ denote the (random) number of good copies of $H$ in $G''$ and let $Z$ be the total number of good copies of $H$ in $G$. Even though we might have accidentally eliminated some good copies of $H$ in $G'$ while forming the subgraph $G''$, it is still true that
  \[
    \Ex[X] \le Z \cdot \binom{n-v_H}{pn-v_H} / \binom{n}{pn} = Z \cdot \binom{pn}{v_H} / \binom{n}{v_H} \le Z \cdot p^{v_H}.
  \]
  Since every copy of $H$ in $G''$ is good and $G''$ has $pn$ vertices, then
  \[
    X \ge e(G'') - \ex(pn, H) \ge e(G'') - M \cdot p^\alpha.
  \]
  Since clearly
  \[
    e(G'') \ge e(G') - \sum_{F \subsetneq H} \sum_{F' \in \cB_F(\cH_i)} \indicator[F' \subseteq G'],
  \]
  and for every $F' \subseteq G$ with $F' \simeq F$, we have $\Pr(F' \subseteq G') = \binom{n-v_F}{pn - v_F} / \binom{n}{pn} \le p^{v_F}$, it follows that
  \begin{equation}
    \label{eq:Z-lower-one}
      Z \cdot p^{v_H} \ge \Ex[X] \ge \Ex[e(G'')] - M \cdot p^\alpha \ge \Ex[e(G')] - \sum_{F \subsetneq H} |\cB_F(\cH_i)| \cdot p^{v_F} - M \cdot p^\alpha.
  \end{equation}
  Finally, if $pn \ge 2r^2$, then
  \[
    \begin{split}
      \Ex[e(G')] & = m \cdot \binom{n-r}{pn-r} / \binom{n}{pn}  = m \cdot \binom{pn}{r} / \binom{n}{r} \ge m \cdot \left(\frac{pn - r}{n}\right)^r = m \cdot p^r \cdot \left(1 - \frac{r}{pn}\right)^r \\
      & \ge m \cdot p^r \cdot \left(1 - \frac{r^2}{pn}\right) \ge \frac{m \cdot p^r}{2},
    \end{split}
  \]
  which substituted into~\eqref{eq:Z-lower-one} yields
  \begin{equation}
    \label{eq:Z-lower-two}
      Z \cdot p^{v_H} \ge \frac{m \cdot p^r}{2} - \sum_{F \subsetneq H} |\cB_F(\cH_i)| \cdot p^{v_F} - M \cdot p^\alpha.
  \end{equation}

  We claim that there is a $p \in [2r^2/n, 1]$ such that $p n$ is an integer and the right-hand side of~\eqref{eq:Z-lower-two} is at least $N \cdot p^{v_H}$, and thus $Z \ge N$. Since $e(\cH_i) = i < N$, this inequality would imply that there is a good copy of $H$ in $G$ that does not belong to $\cH_i$, completing the proof. Hence, it suffices to establish this claim. To this end, note first that by~\eqref{eq:B-F-upper}, we have
  \begin{equation}
    \label{eq:B-F-p-vF-upper}
    \sum_{F \subsetneq H} |\cB_F(\cH_i)| \cdot p^{v_F} \le \frac{m}{4} \cdot \max\left\{ \left((t+1)^3 \cdot \gamma^{t+1}\right)^{e_F-1} \cdot p^{v_F} \colon F \subsetneq H \right\}.
  \end{equation}
  Thus it suffices to have the following three inequalities for every $F \subsetneq H$:
  \begin{eqnarray}
    \label{eq:p-r-vF}
    p^{r - v_F} & \ge & \left((t+1)^3 \cdot \gamma^{t+1} \right)^{e_F-1}, \\
    \label{eq:p-alpha-r}
    p^{\alpha - r} & \le & \frac{\gamma^t}{8} \le \frac{m}{8M}, \\
    \label{eq:p-r-vH}
    p^{r - v_H} & \ge & 8 \left((t+1)^3 \cdot \gamma^{t+1}\right)^{e_H-1} = \frac{8N}{m}.
  \end{eqnarray}
  Indeed, combining inequalities~\eqref{eq:Z-lower-two}, \eqref{eq:B-F-p-vF-upper}, \eqref{eq:p-r-vF}, and~\eqref{eq:p-alpha-r} yields
  \[
    Z \cdot p^{v_H} \ge \frac{m \cdot p^r}{8},
  \]
  which combined with~\eqref{eq:p-r-vH} gives the desired lower bound on $Z$. Note also that both~\eqref{eq:p-r-vF} and~\eqref{eq:p-r-vH} would follow if the following was true for every $F \subsetneq H$:
  \begin{equation}
    \label{eq:p-r-vF-final}
    p^{r - v_F} \ge \left(8 \cdot (t+1)^3 \cdot \gamma^{t+1} \right)^{e_F-1}.
  \end{equation}
  Observe that~\eqref{eq:p-alpha-r} holds trivially for all large enough $t$ if $\alpha = r$. Moreover, \eqref{eq:p-r-vF-final} holds when $v_F = r$, as then $e_F = 1$. Hence, we may assume that $\alpha < r$ and verify~\eqref{eq:p-r-vF-final} only for all $F \subsetneq H$ with $e_F > 1$.

  It is not hard to see that it suffices to show is that for all $F \subsetneq H$ with $e_F > 1$,
  \begin{equation}
    \label{eq:condition-p-final}
    2 \cdot \left(8 \cdot (t+1)^3 \cdot \gamma^{t+1} \right)^{\frac{e_F-1}{v_F-r}} \le \min\left\{ \left(\frac{\gamma^t}{8}\right)^{\frac{1}{r - \alpha}}, \frac{n}{2r^2} \right\}
  \end{equation}
  Indeed, if~\eqref{eq:condition-p-final} holds, then every $p$ in some interval $[p_0, 2p_0] \subseteq [2r^2/n, 1]$ satisfies both~\eqref{eq:p-alpha-r} and~\eqref{eq:p-r-vF-final}. Clearly, this interval contains a $p$ such that $pn$ is an integer. The first of the two inequalities in~\eqref{eq:condition-p-final} holds for all large $t$, as $\gamma > 1$ and by our hypothesis
  \begin{equation}
    \label{eq:alpha-mrH-assumption}
    \frac{e_F-1}{v_F-r} \le m_r(H) < \frac{1}{r-\alpha}.
  \end{equation}
  To see that the second inequality in~\eqref{eq:condition-p-final} holds as well, note first that
  \[
    1 = \ex(r, H) \le M \cdot \left(\frac{r}{n}\right)^\alpha  \le \frac{e(G)}{\gamma^t} \cdot \left(\frac{r}{n}\right)^\alpha \le \frac{n^r}{\gamma^t} \cdot \left(\frac{r}{n}\right)^\alpha,
  \]
  and hence $t \le \frac{r \log n}{\log \gamma}$. It follows that
  \[
    \left(8 \cdot (t+1)^3 \cdot \gamma^{t+1} \right)^{\frac{e_F-1}{v_F-r}} \le \left(\frac{16 \cdot r^{\alpha+3} \cdot \gamma}{(\log \gamma)^3} \cdot n^{r-\alpha}  \cdot (\log n)^3 \right)^{\frac{e_F-1}{v_F-r}} \le \frac{n}{4r^2},
  \]
  provided that $t$ is sufficiently large (and thus $n$ is sufficiently large), since $(r-\alpha) \cdot \frac{e_F-1}{v_F-r} < 1$ by our hypothesis, see~\eqref{eq:alpha-mrH-assumption}. This completes the proof.
\end{proof}

\section{Proofs of Theorems \ref{thm:main-graphs} and \ref{thm:main}}
\label{sec: main theorems}
In this section we prove Theorems \ref{thm:main-graphs} and \ref{thm:main}. Both will be obtained as (more or less) immediate corollaries of our technical Theorem \ref{thm:mainNew}.

\begin{proof}[Proof of Theorem \ref{thm:main}]
  Let $\alpha>r-1/m_r(H)$ and let $A$ be such that
  \[
    \exnH \le An^\alpha
  \]
  for all $n$. Define $M=An^{\alpha}$ and observe that for all $s\in [n]$,
  \[
    \ex(s,H)\le As^{\alpha}=\left(\frac{s}{n}\right)^{\alpha}\cdot An^{\alpha}.
  \]
  Therefore, Theorem~\ref{thm:mainNew} implies the existence of some $C>0$ such that
  \[
    |\FnH|\le 2^{C\cdot \exnH},
  \]
  as claimed.
\end{proof}

Using a standard probabilistic argument, one can show that for every $r$-uniform hypergraph $H$ with at least two edges, the bound $\exnH \ge c_Hn^{r-1/m_r(H)}$ holds for some positive constant $c_H$.  In particular, if $\exnH \le An^\alpha$ for all $n$, as in the statement of Theorem~\ref{thm:main}, then $\alpha \ge r - 1/m_r(H)$. It turns out that when $H$ is a graph that contains a cycle, the stronger lower bound
\begin{equation}
  \label{eq:Bohman-Keevash}
  \exnH \ge c_H n^{2-1/m_2(H)} (\log n)^{1/(e_H-1)}
\end{equation}
holds for all $n$. This was first proved by Bohman and Keevash~\cite{BoKe10} and later generalized to hypergraphs of higher uniformity by Bennett and Bohman~\cite{BeBo16}.

\begin{proof}[Proof of Theorem~\ref{thm:main-graphs}]
  Suppose that $H$ contains a cycle and $\alpha$ and $A$ are such that $\exnH \le An^{\alpha}$ for all $n$. It follows from~\eqref{eq:Bohman-Keevash} that $\alpha > 2-1/m_2(H)$. The assertion of the theorem now easily follows from Theorem~\ref{thm:main}.
\end{proof}

In Appendix~\ref{sec:lower-bounds-FnH}, we revise an old argument of Kohayakawa, Kreuter, and Steger~\cite{KoKrSt98} to derive a stronger form of~\eqref{eq:Bohman-Keevash} from a version of the famous result of Ajtai, Koml\'os, Pintz, Spencer, and Szemer\'edi~\cite{AjKoPiSpSz82} due to Duke, Lefmann, and R\"odl~\cite{DuLeRo95}.

\section{Proofs of Theorems \ref{thm:infiniteSequence} and \ref{thm:infiniteSequenceHyper}}
\label{sec:infinite}

In this section, we prove Theorems~\ref{thm:infiniteSequence} and~\ref{thm:infiniteSequenceHyper}. That is, we show that if $\exnH$ exceeds the standard probabilistic lower bound of $c_H n^{r-1/m_r(H)}$ by a factor polynomial in $n$, then Theorem~\ref{thm:mainNew} implies that $|\FnH| \le 2^{C \cdot \exnH}$ for infinitely many $n$. Since Theorem~\ref{thm:infiniteSequence} is simply the case $r=2$ of Theorem~\ref{thm:infiniteSequenceHyper}, we only prove the latter.

\begin{proof}[Proof of Theorem~\ref{thm:infiniteSequenceHyper}]
  Let $H$ be an $r$-uniform hypergraph and suppose that there is an $\eps > 0$ such that
  \[
    \exnH \ge n^{r-1/m_r(H)+\eps}
  \]
  for all $n$. We shall construct an infinite sequence of $n$ satisfying the hypothesis of Theorem~\ref{thm:mainNew} with $M = \exnH$ and $\alpha = r-1/m_2(H) + \eps/2$. Then, for each $n$ in the sequence, we obtain
  \[
    |\FnH|\le 2^{C\cdot \exnH}
  \]
  for some $C$ that depends only on $\alpha$, $\eps$, and $H$. This will complete the proof.

  Assume towards a contradiction that there are only finitely many $n$ satisfying the hypothesis of Theorem~\ref{thm:mainNew}. In particular, there exists an $N$ such that for all $n_0\ge N$,
  \[
    \ex(n_1,H) > \ex(n_0,H) \cdot \left(\frac{n_1}{n_0}\right)^{\alpha}
  \]
  for some $n_1 < n_0$. Choose as small $\delta > 0$, let $n_0 = \lceil N^{1/\delta} \rceil$, and suppose that we have defined $n_0, \dotsc, n_{k-1}$ this way. If $n_{k-1} \ge n_0^\delta \ge N$, then there is some $n_k \le n_{k-1} - 1 \le n_0 - k$ such that
  \[
    \ex(n_k, H) > \ex(n_{k-1}, H) \cdot \left(\frac{n_k}{n_{k-1}}\right)^{\alpha} > \ex(n_0,H) \cdot \left(\frac{n_k}{n_{k-1}}\right)^\alpha \cdot \left(\frac{n_{k-1}}{n_0}\right)^\alpha.
  \]
  Note that if $n_k \le n_0^{\delta}$, then the lower bound $\exnH\ge \eps n^{r-1/m_r(H)+\eps}$ implies
  \[
    n_0^{r\delta} \ge n_k^r > \binom{n_k}{r}\ge \ex(n_k, H) > \ex(n_0,H) \cdot n_0^{(\delta - 1)\alpha} \ge \eps n_0^{\alpha+\eps/2} \cdot n_0^{(\delta-1)\alpha} \ge \eps n_0^{\alpha \delta + \eps/2}.
  \]
  This is clearly impossible, as $\eps > 0$ is fixed and we may choose $N$ as large as we want and $\delta$ as small as we want. Therefore, there must be some $n \ge N$ for which the hypothesis of Theorem~\ref{thm:mainNew} holds, a contradiction.
	\end{proof}

\section{Concluding remarks}
\label{sec:remarks}

In order to prove that $|\FnH| \le 2^{C\exnH}$ for some $r$-uniform hypergraph $H$ using Theorem~\ref{thm:mainNew}, one needs to assume that $\exnH \ge \eps n^{r-1/m_r(H)+\eps}$ for some positive constant $\eps$. Indeed, one clearly needs $M = O(\exnH)$ and letting $s = r$ in the hypothesis of the theorem yields the bound $M = \Omega(n^\alpha)$, where $\alpha$ is a constant satisfying $\alpha > r-1/m_r(H)$. Some kind of separation of $\exnH$ from $n^{r-1/m_r(H)}$ is crucial for our approach to work. Even though one could most likely allow $\eps$ to tend to $0$ with $n$ at some rate, the lower bound on $|\FnH|$ proved by Corollary~\ref{cor:FnH-better-lower} shows that if $|\FnH| \le 2^{C\exnH}$, then the ratio of $\exnH$ to $n^{r-1/m_r(H)}$ has to be at least $(\log n)^{1+c_H}$ for some $c_H > 0$. The reason why we have not tried to weaken this separation assumption is that we believe that the following is true.

\begin{conj}
  \label{conj:separation}
  Let $H$ be an arbitrary graph that is not a forest. There exists an $\eps > 0$ such that $\exnH \ge \eps n^{2-1/m_2(H)+\eps}$ for all $n$.
\end{conj}

Note that Conjecture~\ref{conj:separation} is weaker than Conjecture~\ref{conj:weak}, since if $\exnH = \Theta(n^\alpha)$ for some constant $\alpha$, then necessarily $\alpha > 2-1/m_2(H)$ by Proposition~\ref{prop:ex-lower-bound-log}. Even though we believe that Conjecture~\ref{conj:separation} is interesting in its own right, additional motivation for it stems from Theorem~\ref{thm:infiniteSequence} -- any graph $H$ for which the conjecture holds has the property that $|\FnH| \le 2^{C\exnH}$ for infinitely many $n$.

\medskip
\noindent
\textbf{Acknowledgment.}
We would like to thank Misha Tyomkyn for a helpful discussion on degenerate hypergraph problems.

\bibliographystyle{amsplain}
\bibliography{bipartite-free}

\appendix

\section{Lower bounds for the number of $H$-free hypergraphs}
\label{sec:lower-bounds-FnH}

Recall that a hypergraph $\cH$ is \emph{linear} if every pair of distinct edges of $\cH$ intersects in at most one vertex. We shall use the following version of the famous result of Ajtai, Koml\'os, Pintz, Spencer, and Szemer\'edi~\cite{AjKoPiSpSz82} due to Duke, Lefmann, and R\"odl~\cite{DuLeRo95} to derive a lower bound on the extremal number of strictly $r$-balanced $r$-uniform hypergraphs with at least three edges.

\begin{thm}[\cite{DuLeRo95}]
  Let $k \ge 3$ and let $\cH$ be a $k$-uniform hypergraph with $\Delta(\cH) \le D$. If $\cH$ is linear, then
  \[
    \alpha(\cH) \ge c\cdot v(\cH) \cdot \left(\frac{\log D}{D} \right)^{\frac{1}{k-1}}
  \]
  for some constant $c$ that depends only on $k$.
\end{thm}

Recall that an $r$-uniform hypergraph $H$ is \emph{strictly $r$-balanced} if for every $F \subsetneq H$ with at least two edges,
\[
  \frac{e_F - 1}{v_F - r} < \frac{e_H - 1}{v_H - r}.
\]
Our result (and its proof) is a fairly straightforward generalization of~\cite[Theorem 8]{KoKrSt98}, which established the same result in the special case when $H$ is an even cycle. Before stating the result, we need the following definitions. Let $G^{(r)}_{n,m}$ denote an $r$-uniform hypergraph on $n$ vertices with exactly $m$ edges, chosen uniformly at random. Moreover, given two hypergraphs $G$ and $H$, we denote by $\ex(G,H)$ the maximum number of edges in an $H$-free subhypergraph of $G$.

\begin{proposition}
  \label{prop:ex-lower-bound-log}
  Let $r \ge 2$ and suppose that $H$ is a strictly $r$-balanced $r$-uniform hypergraph with at least three edges. There exists a $\delta > 0$ such that for all $m \ge n^{r-1/m_r(H)}$, with probability at least one half,
  \[
    \ex(G_{n,m}^{(r)}, H) \ge \delta n^{r-1/m_r(H)} \left( \log \left(\frac{m}{n^{r-1/m_r(H)}}\right) \right)^{\frac{1}{e_H-1}}.
  \]
\end{proposition}

It is very likely that the case $r=2$ of Proposition~\ref{prop:ex-lower-bound-log} is implicit in the work of Bohman and Keevash~\cite{BoKe10}; one can stop the $H$-free process after only $m$ edges have been considered for addition. Moreover, it is possible that the general case can be proved using the techniques of Bennett and Bohman~\cite{BeBo16}, although one cannot invoke~\cite[Theorem~1.1]{BeBo16} directly, as this would require every edges of $G_{n,m}^{(r)}$ to be contained in the same number of copies of $H$. Finally, note that the assumption $e_H \ge 3$ is crucial. Indeed, suppose that $e_H = 2$ and the two edges of $H$ intersect in $\ell$ vertices. Then $m_r(H) = 1/(r-\ell)$ and it follows from the pigeonhole principle  that $\binom{r}{\ell} \cdot \exnH \le \binom{n}{\ell}$, proving that $\exnH \le n^{\ell} = n^{r-1/m_r(H)}$.

\begin{proof}
  By choosing $\delta$ small, we may assume that $n$ is sufficiently large. Moreover, given the dependence on $m$ of the claimed lower bound, we may assume that $m \le n^{r - 1/m_r(H) + \eps}$ for some small positive $\eps$ which we will specify later (for larger values of $m$, all we need to do is to multiply the obtained bound by a constant factor, depending only on $\varepsilon$). Let $\cH$ be the (random) $e_H$-uniform hypergraph whose vertex set is the edge set of $G_{n,m}^{(r)}$ and whose edges are the edge sets of all copies of $H$ in $G_{n,m}^{(r)}$. Clearly, $\ex(G_{n,m}, H)$ is just the independence number of $\cH$. Let $p = m/\binom{n}{r}$. Even though $\cH$ is not necessarily linear, we shall show that it contains an induced subhypergraph $\cH'$ with at least $m/2$ vertices that is linear and satisfies $\Delta(\cH') \le C n^{v_H-r}p^{e_H-1}$. This will allow us to conclude that
  \[
    \alpha(\cH) \ge \alpha(\cH') \ge c \frac{m}{2} \cdot \left(\frac{\log \left(Cn^{v_H-r}p^{e_H-1}\right)}{C n^{v_H-r} p^{e_H-1}}\right)^{\frac{1}{e_H-1}} \ge c' n^{r - 1/m_r(H)} \cdot \left(\log \left(\frac{m}{n^{r -1/m_r(H)}}\right) \right)^{\frac{1}{e_H-1}},
  \]
  where we have used
  \[
    n^{v_H-r}p^{e_H-1} = \Theta\left(\left(\frac{m}{n^{r-1/m_r(H)}}\right)^{e_H-1}\right).
  \]
  We find such an $\cH'$ using a simple deletion argument. We first compute the expected number of pairs of edges of $\cH$ that intersect in more than one vertex. We claim that
  \[
    \Ex[\text{\#pairs}] \le \sum_{\substack{F \subsetneq H \\ e(F) \ge 2}} p^{2e_H-e_F} n^{2v_H - v_F} \ll pn^r.
  \]
  The last inequality would follow if we showed that for every proper subgraph $F \subsetneq H$ with at least two edges,
  \begin{equation}
    \label{eq:mrH-vs-mrF}
    \left( p^{e_H} n^{v_H} \right)^2 \ll pn^r \cdot p^{e_F} n^{v_F}.
  \end{equation}
  Inequality~\eqref{eq:mrH-vs-mrF} holds provided that $\eps > 0$ is sufficiently small. Indeed, letting $p = n^{-\beta}$, one can check that~\eqref{eq:mrH-vs-mrF} is equivalent to
  \[
    2 \cdot \big(v_H - r - \beta \cdot (e_H-1)\big) < v_F-r-\beta \cdot (e_F-1).
  \]
  Moreover, when $\beta = 1/m_r(H) = (v_H-r) / (e_H-1)$, then the left-hand side is zero and our assumption that $m_r(F) < m_r(H)$ implies the right-hand side is positive. Thus, with probability at least $3/4$, we can delete one vertex of $\cH$ from every such pair, obtaining a linear induced subhypergraph $\cH''$ of $\cH$ with at least $3m/4$ vertices, say.

  Finally, fix any $A \in\binom{n}{r}$. By symmetry,
  \[
    \Ex[\deg_{\cH}A] = \Pr[A\in V(\cH)]\cdot \Ex[\text{average degree of }\cH] \le p\cdot \frac{e_H n^{v_H}p^{e_H}}{m} \le p \cdot \frac{Cn^{v_H-r}p^{e_H-1}}{16},
  \]
  provided that $C$ is a sufficiently large constant. By Markov's inequality,
   \[
   \Pr[\deg_{\cH}A > Cn^{v_H - r}p^{e_H-1}] < p/16
  \]
  By Markov's inequality, with probability at least $3/4$, the hypergraph $\cH$ contains at most $m/4$ vertices of degree exceeding $C n^{v_H-r} p^{e_H-1}$. In particular, with probability at least one half, we may delete them from $\cH''$ to obtain a linear induced subhypergraph $\cH'$ of $\cH$ with at least $m/2$ vertices and maximum degree at most $C n^{v_H-r} p^{e_H-1}$.
\end{proof}

\begin{cor}
  \label{cor:FnH-better-lower}
  Let $r \ge 2$ and suppose that $H$ is a strictly $r$-balanced $r$-uniform hypergraph with at least three edges. There exists a positive constant $c$ such that
  \[
    |\FnH| \ge \exp\left(c n^{r-1/m_r(H)} (\log n)^{\frac{e_H}{e_H-1}} \right).
  \]
\end{cor}
\begin{proof}
  Let $\gamma = (2m_r(H))^{-1}$, let $m = n^{r - 1/m_r(H) + \gamma}$, and let $\delta$ be the constant from the statement of Proposition~\ref{prop:ex-lower-bound-log}. The proposition implies that with probability at least one half,
  \begin{equation}
    \label{eq:exGnm-lower-bound}
    \ex(G_{n,m}^{(r)}, H) \ge \delta n^{r-1/m_r(H)} \left( \log \left(\frac{m}{n^{r-1/m_r(H)}}\right) \right)^{\frac{1}{e_H-1}} = \delta' n^{r-1/m_r(H)}  \left( \log n \right)^{\frac{1}{e_H-1}},
  \end{equation}
  where $\delta' = \delta \cdot \gamma^{1/(e_H-1)} > 0$. Denote the right-hand side of~\eqref{eq:exGnm-lower-bound} by $m'$. We have just shown that at least a half of all $r$-uniform hypergraphs with vertex set $[n]$ and $m$ edges contain an $H$-free subhypergraph with $m'$ edges. Now, a straightforward double-counting argument gives
  \[
    \begin{split}
      \left| \FnH \right| & \ge \frac{1}{2} \cdot \binom{\binom{n}{r}}{m} \cdot \binom{\binom{n}{r}-m'}{m-m'}^{-1} = \frac{1}{2} \cdot \binom{\binom{n}{r}}{m'} \cdot \binom{m}{m'}^{-1} \ge \frac{1}{2} \cdot \left(\frac{n^r}{2 \cdot r! \cdot m'}\right)^{m'} \cdot \left(\frac{m'}{em}\right)^{m'} \\
      & = \frac{1}{2} \cdot \left(\frac{n^\gamma}{2e \cdot r!}\right)^{m'} \ge \exp\left(c' m' \log n\right) = \exp\left(c n^{r-1/m_r(H)} (\log n)^{\frac{e_H}{e_H-1}} \right),
    \end{split}
  \]
  as claimed.
\end{proof}

\end{document}